\documentclass[12pt]{amsart}
\usepackage{amssymb,amsmath,amsfonts,latexsym}
\usepackage{bm}
\usepackage{array,graphics,color}
\newcommand{\newsection}[1]{\setcounter{equation}{0} \section{#1}}
\numberwithin{equation}{section}

\newtheorem{propn}{Proposition}[section]
\newtheorem{thm}[propn]{Theorem}
\newtheorem{lemma}[propn]{Lemma}

\newtheorem{remark}[propn]{Remark}
\newtheorem{remarks}[propn]{Remarks}
\newtheorem*{thm*}{Theorem}
\theoremstyle{definition}

\newcommand{\Hil}{\mathcal{H}}

\newcommand{\Z}{\mathbb{Z}_+}

\newcommand{\T}{\mathcal{T}}

\newcommand{\h}{\hat{T}}

\newcommand{\D}{\mathbb{D}}

\DeclareMathOperator{\ran}{ran}

\newcommand{\clb}{\mathcal{B}}

\newcommand{\cld}{\mathcal{D}}
\newcommand{\cle}{\mathcal{E}}
\newcommand{\clf}{\mathcal{F}}

\newcommand{\clh}{\mathcal{H}}

\newcommand{\clk}{\mathcal{K}}
\newcommand{\cll}{\mathcal{L}}

\newcommand{\clq}{\mathcal{Q}}

\newcommand{\clt}{\mathcal{T}}

\newcommand{\clx}{\mathcal{X}}
\newcommand{\cly}{\mathcal{Y}}

\newcommand{\z}{\bm{z}}
\newcommand{\w}{\bm{w}}
\newcommand{\raro}{\rightarrow}

\newcommand{\fbn}{\mathfrak{B}^n_{1,n}}

\begin{document}
	
	\title[{Dilation under Brehmer positivity}]{Isometric dilations of commuting contractions and Brehmer positivity}

	\author[Barik]{Sibaprasad Barik}
	\address{Department of Mathematics, Indian Institute of Technology Bombay, Powai, Mumbai, 400076, India}
	\email{sibaprasadbarik00@gmail.com}
	
	\author[Das] {B. Krishna Das}
	\address{Department of Mathematics, Indian Institute of Technology Bombay, Powai, Mumbai, 400076, India}
	\email{dasb@math.iitb.ac.in, bata436@gmail.com}
	
	\begin{abstract}
It is well-known that an $n$-tuple $(n\ge 3)$ of commuting contractions does not posses an isometric dilation, in general. Considering a class of 
$n$-tuple of commuting contractions satisfying certain positivity assumption, we construct their isometric dilations
and consequently establish their von Neumann inequality. The positivity 
assumption is related to Brehmer positivity and motivated 
by the study of isometric dilations of operator tuples 
in ~\cite{BDHS}.

	\end{abstract}

\maketitle

\newsection{Introduction}
Starting point of the dilation theory is the result of Sz.-Nagy which says that a Hilbert space contraction always dilates to an isometry acting on a bigger Hilbert space. More precisely, Sz.-Nagy proved the following.
\begin{thm}[\cite{NF}] Let $T$ be a contraction on a Hilbert space $\clh$. 
Then there exists a Hilbert space $\clk$ $(\clk\supset\clh)$ and an isometry $V\in\clb(\clk)$ such that
\[
T^k=P_{\clh}V^k|_{\clh}
\]
for all $k\in\mathbb Z_{+}$, where $P_{\clh}$ denotes the orthogonal projection in $\clb(\clk)$ with range $\clh$.
\end{thm}	
\noindent This result is the cornerstone of the extremely useful and extensive theory of Sz.-Nagy and Foias on single contractions ~\cite{NF}. Ando~\cite{An}, generalizing Sz.-Nagy's dilation, constructed isometric dilations for pairs of commuting contractions. In other words, he constructed a pair of commuting isometries $(V_1,V_2)$ on $\clk\supseteq \clh$ corresponding 
to each pair of commuting contractions $(T_1,T_2)$ on $\clh$ such that
 \[
T_1^{m_1}T_2^{m_2}=P_{\clh}V_1^{m_1}V_2^{m_2}|_{\clh}
\]
for all $(m_1,m_2)\in\mathbb Z_{+}^2$.
Producing some counterexamples Parrott \cite{Par} showed that for  $n\geq3$, $n$-tuples of commuting contractions do not posses isometric dilations in general. This leads us to study $n$-tuples $(n\ge 3)$ of commuting contractions more closely and identify $n$-tuples of commuting contractions which posses isometric dilations. For the rest of this article, we assume that $n\ge 3$. We denote by $\T^n(\clh)$ the
 set of all $n$-tuple of commuting contractions, that is
\[
\T^n(\clh)=\{(T_1,\ldots,T_n): T_i\in\clb(\clh),  \|T_i\| \leq
1, T_i T_j = T_j T_i, 1 \leq i, j \leq n\},
\]
and for $T=(T_1,\ldots,T_n)\in\T^n(\clh)$ and $\bm{k}=(k_1,\ldots,k_n)\in\Z^n$ we define $T^{\bm{k}}=T_1^{k_1}\cdots T_n^{k_n}$. 
For Hilbert spaces $\clh$ and $\clk$ with $\clk\supset \clh$, 
an $n$-tuple of commuting isometries (respectively unitaries) $V\in\T^n(\clk)$ is called an isometric dilation (respectively unitary dilation) of $T\in\T^n(\clh)$ if
\[
T^{\bm{k}}=P_{\clh}V^{\bm{k}}|_{\clh}
\]
for all $\bm{k}\in\Z^n$. There is a growing literature exploring the interplay between dilation and positivity, and several classes of operator tuples 
under certain positivity assumptions are known to have isometric 
dilations.  See \cite{Ag}, \cite{Bre}, \cite{VV}, \cite{MV},  \cite{CV1}, \cite{CV2},  \cite{GS},  \cite{DS},  \cite{DSS},  \cite{BDHS} and \cite{BDS} for more details in the polydisc setup. The one which is relevant for us is the dilation of operator tuples under Brehmer positivity ~\cite{Bre}. An $n$-tuple of commuting contraction $T\in\T^n(\clh)$ is said to satisfy Brehmer positivity if
\[
\sum_{F\subset G}
(-1)^{|F|} T_{F} T_{F}^*\geq0
\]
for all $G\subset\{1,\ldots,n\}$, where $T_F:=T_{n_1}\cdots T_{n_k}$
for any $F=\{n_1,\ldots,n_k\}\subset\{1,\dots,n\}$ and by convention 
$T_{\emptyset}=I_{\clh}$.
We denote by $\mathfrak{B}^n(\clh)$ the class of $n$-tuples of commuting contractions on $\clh$ satisfying Brehmer positivity, that is 
\[
\mathfrak{B}^n(\clh):=\{T\in\T^n(\clh): T \,\text{satisfies Brehmer positivity}\}.
\]
It is clear from the definition that if $T\in \mathfrak{B}^n(\clh)$ then 
\[
(T_{n_1},\dots, T_{n_k})\in \mathfrak{B}^k(\clh)
\]
for any non-empty subset $\{n_1,\ldots,n_k\}$ of $\{1,\dots,n\}$.
It has been shown in ~\cite{Bre} that every $T\in \mathfrak{B}^n(\clh)$ possesses  
isometric dilation (also see ~\cite{NF}). In fact, more stronger result is true. Namely, 
$T\in \mathfrak{B}^n(\clh)$ if and only if $T$ has a $*$-regular unitary dilation $U$,
that is
\[
T^{\bm\alpha_{+}} T^{* \bm\alpha_{-}}=P_{\clh}U^{* \bm \alpha_{-}}U^{\bm \alpha_{+}}|_{\clh} 
\] 
for all  $\bm\alpha\in \mathbb Z^n$, where $\bm\alpha_{+} =(\alpha_1^{+},\dots,\alpha_n^{+}), \bm\alpha_{-}=(\alpha_1^{-},\dots,\alpha_n^{-})\in \mathbb Z_{+}^n$ and $\alpha_i^{+}:=\mbox{max}\{\alpha_i,0\}$, 
$\alpha_i^{-}:=\mbox{max}\{-\alpha_i,0\}$. 
The aim of this article is to exhibit a class of $n$-tuple of commuting 
contractions having isometric dilations so that the class is larger than the class of $n$-tuples of commuting contractions satisfying Brehmer positivity. To describe the class of operator tuples
 succinctly, we adopt the following notation.
 For $T=(T_1,\dots,T_n)\in\T^n(\clh)$ and  $1 \leq i \leq n$, we define
\[
\hat{T}_{i} := (T_1,\dots,T_{i-1}, T_{i+1},\dots,T_{n})
\in \T^{n-1}(\clh),
\]
the $(n-1)$-tuple obtained from $T$ by deleting $T_i$.
The class we consider in this article is
denoted by $\mathfrak{B}^n_{p,q}(\clh)$ for some 
$1\leq p<q\leq n$ and defined as 
\[
\mathfrak{B}^n_{p,q}(\clh):=\{T=(T_1,\ldots,T_n)\in\clt^n(\clh):\, \hat{T}_p, \hat{T}_{q}\in \mathfrak{B}^{n-1}(\clh)\}.
\]	
From the definition it is clear that $\mathfrak{B}^n(\clh)\subset \mathfrak{B}^n_{p,q}(\clh)$ and the containment can be shown to be proper.
By means of an explicit construction, we show that 
$T\in \mathfrak{B}^n_{p,q}(\clh)$ if and only if $T$ 
has  an isometric dilation $V$ such that $\hat{V}_p$ and $\hat{V}_q$ are $*$-regular 
isometric dilation of $\hat{T}_p$ and $\hat{T}_q$, respectively. 
For the base case ($n=3$), the existence of isometric dilations for $\mathfrak{B}^3_{1,3}(\clh)$ is obtained by Gasper and Suciu in \cite[Theorem 12]{GS} . However our proof, for this particular case, is completely different than that of
~\cite{GS} and is based on an explicit construction of dilating isometries.
The present consideration is also motivated by ~\cite{BDHS} where 
the authors considered the following class of operator tuples  
\begin{equation}\label{tnpq}
\clt^n_{p,q}:=\{T\in \clt^n(\clh): \hat{T}_p\ \text{ and } \hat{T}_q\ 
\text{ satisfy Szeg\"o positivity  and } \hat{T}_q\ \text{ is pure}\}
\end{equation}
and found their isometric dilations explicitly.
We say that $T\in \clt^n(\clh)$ satisfies Szeg\"o positivity if 
\[
\sum_{F\subset \{1,\dots,n\}}
(-1)^{|F|} T_{F} T_{F}^*\geq0,
\]
and we say $T$ is pure if $T_i^{* m}h\to 0$ as $m\to\infty$ for all 
$h\in\clh$ and $i=1,\dots,n$. It is easy to see
that if $T \in\clt^n(\clh)$ is pure and satisfies Szeg\"o positivity 
then it satisfies Brehmer positivity. Thus it is evident 
that if $T\in \clt^n_{p,q}(\clh)$ and $T$ is pure then 
$T\in \mathfrak{B}^n_{p,q}(\clh)$. From this point of view,
 the present work is also 
a generalization of ~\cite{BDHS}.

An added benefit of this consideration is the von Neumann inequality 
for the class $\mathfrak{B}^n_{p,q}(\clh)$. If $T\in \clt^n(\clh)$ has an 
isometric dilation, then it is immediate that $T$ satisfies the 
von Neumann inequality, that is for all
 $p\in \mathbb{C}[z_1, \ldots, z_n]$	
\[
\|p(T)\|\le \mbox{sup}_{\z\in\D^n}|p(\z)|,
\]
where $\mathbb{D}^n$ is the open unit $n$-polydisc in $\mathbb{C}^n$. Thus, as an immediate 
consequence of our isometric dilations, we obtain that each tuple in $\mathfrak{B}^n_{p,q}(\clh)$ satisfies von Neumann inequality. 
It is worth mentioning here that von Neumann inequality does
 not hold in general for $n$-tuple of commuting contractions 
 (see \cite{V} and \cite{CD1}). 
 More details on von Neumann inequality for $n$-tuple of commuting contractions can be found in \cite{CD}, \cite{Ho1}, \cite{Ho2}, \cite{K1} and \cite{Pi}.

Rest of the paper is organized as follows. In the next section we develop some background material and state some known results which are relevant in the present context. Section ~\ref{class} deals with isometric dilations and von Neumann inequality for the class $\mathfrak{B}^n_{p,q}(\clh)$.

\section{Preliminaries}
 
In most of the cases, isometric dilations of tuple of commuting contractions are isometric co-extensions.  However, we use 
co-extension as an intermediate step to obtain isometric dilations in the 
present context.
  Let $\clh$ and $\clk$ be Hilbert spaces, and let $T \in \T^n(\clh)$
 and $V$ be an $n$-tuple of contractions on $\clk$. We say that $V$
 is a \textit{co-extension} of $T$ if there exists an isometry $\Pi: \clh
 \raro \clk$ such that
 \begin{equation}\label{dileq}
 \Pi T_i^{*} = V_i^{*} \Pi,
 \end{equation}
 for all $i=1,\ldots,n$. In addition,
 if 
 \[
 \clk=\overline{\mbox{span}}\{V^{\bm k}h:  \bm k\in \mathbb Z_{+}^n, h\in 
 \mbox{ran}\, \Pi\},
 \]
 then we say  $V$ is a \emph{minimal co-extension} of $T$.
 We warn the reader here that we do not require $V$ to be a commuting 
 tuple. 
 Let $\clq=\ran\Pi$. Then, by equation (\ref{dileq}),
 \[
 \Pi T_i \Pi^*=(\Pi T_i \Pi^*)(\Pi \Pi^*)
 =(\Pi \Pi^* V_i)(\Pi \Pi^*) 
 =P_{\clq}V_i|_{\clq}.
 \]
This implies that $(P_{\clq}V_1|_{\clq},\ldots,P_{\clq}V_n|_{\clq})$ 
is a commuting tuple of contractions even if $V$ is not an commuting tuple of contractions. Moreover,
$(T_1,\ldots,T_n)$ is unitary equivalent to $(P_{\clq}V_1|_{\clq},\ldots,P_{\clq}V_n|_{\clq})$.
Also, 
\[
(P_{\clq}V_i|_{\clq})^*=\Pi T_i^*\Pi^*=V_i^*\Pi\Pi^*=V_i^*|_{\clq},
\]
which implies, $(V_1,\ldots,V_n)$ is a co-extension of $(P_{\clq}V_1|_{\clq},\ldots,P_{\clq}V_n|_{\clq})$ $(\cong (T_1,\ldots,T_n))$.

For a Hilbert space $\cle$, the $\cle$-valued Hardy space over $\D^n$
is denoted by $H_{\cle}^2(\D^n)$ and defined as the space of all $\cle$-valued analytic functions $f=\sum_{\bm{k}\in\Z^n}a_{\bm{k}}\z^{\bm{k}}$ $(a_{\bm{k}}\in\cle)$ on $\D^n$ such that
\[
\sum_{\bm{k}\in\Z^n}\|a_{\bm{k}}\|^2<\infty.
\]
The space $H_{\cle}^2(\D^n)$ is a reproducing kernel Hilbert space with kernel $\mathbb{S}_nI_{\cle}$ where $\mathbb{S}_n$ is the Szeg\"o kernel on the polydisc $\D^n$ given by 
\[
\mathbb{S}_n(\z,\w)=\prod_{i=1}^n(1-z_i\bar{w_i})^{-1} \quad(\z, \w \in \D^n).
\]
The $n$-tuple of shifts $(M_{z_1}, \ldots, M_{z_n})$ on $H^2_{\cle}(\D^n)$ is defined by 
\[
(M_{z_i} f)(\w) = w_i f(\w)\quad(\ \bm{w}\in\D^n),
\]
for all $i=1,\dots,n$ and is a tuple of commuting isometries.
For $T\in\T^n(\clh)$, we say $T\in\T^n(\clh)$ satisfies Szeg\"o positivity if $\mathbb{S}_n^{-1}(T,T^*)\geq0$, where
\[
\mathbb{S}_n^{-1}(T,T^*):=\sum_{F\subset\{1,\ldots,n\}}
(-1)^{|F|} T_{F} T_{F}^*.
\]
In such a case, we define the defect operator and the defect spaces as 
\[
	D_T:=\mathbb{S}_n^{-1}(T,T^*)^{1/2}\quad\text{and}\quad\cld_T:=\overline{\ran}\,\,\mathbb{S}_n^{-1}(T,T^*),
\]
respectively. The map  $\Pi: \clh\to H^2_{\cld_T}(\D^n)$ defined by
	\begin{equation}\label{szegodil}
	(\Pi h)(\z)=\sum_{\bm{k}\in\Z^n}\z^{\bm{k}}\otimes D_T T^{*\bm{k}}h\quad(\z\in\D^n)
	\end{equation}
is called the \textit{canonical dilation map} corresponding 
to a Szeg\"o  tuple $T\in\clt^n(\clh)$ and it satisfies 
\[
	\Pi T_i^*=M_{z_i}^*\Pi
	\]
	for all $i=1,\ldots,n$, and 
	\[
	\|\Pi h\|^2=\lim_{\bm{k}\to\infty}\sum_{F\subset\{1,\ldots,n\}}(-1)^{|F|}\|T_F^{* \bm{k}}h\|^2
	\]
	for all $h\in\clh$. In addition, if $T$ is pure then by the above identity
	$\Pi$ is an isometry and therefore the $n$-tuple of shift $(M_{z_1},\dots,M_{z_n})$ on $H^2_{\cld_T}(\D^n)$ is a co-extension of $T$. In particular, a pure Szeg\"o
$n$-tuple $T$ dilates to the $n$-tuple of shifts on $H^2_{\cld_T}(\D^n)$ (see ~\cite{CV1} and ~\cite{MV}).  A pure Szeg\"o tuple 
$T\in \clt^n(\clh)$ is also a member of $\mathfrak{B}^n(\clh)$, that is 
$T$ satisfies Brehmer positivity. Thus the class $\mathfrak{B}^n(\clh)$
is larger than the class of pure and Szeg\"o $n$-tuples on $\clh$. 
Recall that an isometric dilation $V\in\clt^n(\clk)$ of $T\in\clt^n(\clh)$ is 
$*$-regular (respectively regular) if it satisfies  
\[
P_{\clh}V^{*\bm\alpha_-}V^{\bm\alpha_+}|_{\clh}= T^{\bm\alpha_+}T^{* \bm\alpha_-}\quad(\bm\alpha\in\mathbb{Z}^n)
\]
or respectively,
\[
P_{\clh}V^{* \bm\alpha_-}V^{\bm\alpha_+}|_{\clh}= T^{* \bm\alpha_-}T^{\bm\alpha_+}\quad(\bm\alpha\in\mathbb{Z}^n).
\]
It is well-known that $T\in \mathfrak{B}^n(\clh)$ if and only if $T$ has a $*$-regular isometric dilation (see~\cite{NF}). For explicit constructions of $*$-regular isometric dilations for the class $\mathfrak{B}^n(\clh)$ see ~\cite{AM} and ~\cite{Tim}.
Simply taking adjoint, it is evident that  $T^*\in \mathfrak{B}^n(\clh)$ if and only if 
$T$ has a regular isometric dilation. For example, if $(V,T)$ is a pair 
of commuting contractions on $\clh$ and if $V$ is an isometry then 
$(V^*,T^*)\in \mathfrak{B}^2(\clh)$ and therefore $(V,T)$ has a regular isometric dilation. We end the section with a lemma which will be useful 
for us. The lemma may be known to experts and we are unable to find any suitable reference for it, and that is why we include a proof. We say that an operator tuple $T\in\clt^n(\clh)$ is doubly commuting if $T_iT_j^*=T_j^*T_i$ for all 
$1\le i<j\le n$.
\begin{lemma}\label{doubly commuting lifting}
Let $(W_1,W_2)$ on $\clk$ be a minimal regular unitary dilation 
of $(T_1,T_2)$ on $\clh$. Suppose $(S_1,\dots, S_d)$ is a tuple of isometries on $\clh$ such that $(T_1, S_1,\dots, S_d)$ and
 $(T_2,S_1,\dots, S_d)$ are doubly commuting. Then there exists 
 a commuting tuple of isometries $(U_1,\dots, U_d)$ on $\clk$ such 
 that $(W_1,W_2, U_1,\dots, U_d)$ is an isometric dilation of the tuple 
 $(T_1,T_2, S_1,\dots, S_d)$ and that $(W_1, U_1,\dots, U_d)$ and 
 $(W_2,U_1,\dots, U_d)$ are doubly commuting.
  
 Moreover, $(U_1,\dots, U_d)$ is both an extension and a co-extension 
 of $(S_1,\dots, S_d)$.
\end{lemma}
\begin{proof}
	Let $W=(W_1,W_2)$. Since $W$ on $\clk$ is the minimal unitary dilation of $T=(T_1,T_2)$,  
	\begin{equation}\label{minimality}
	\clk=\overline{\mbox{span}}\{W^{\bm k}h:  \bm k\in \mathbb Z^2, h\in 
	\clh\}.
	\end{equation}
For $1\le i\le d$, $h_1,\dots,h_r\in \clh$ and $\bm k^1,\dots,\bm k^r\in\mathbb Z^2$, we have 
\begin{align*}
\big\|\sum_{l=1}^r W^{\bm k^l}S_i h_l\big\|^2& =\sum_{l,m=1}^r\langle 
W^{ *\bm k^m}W^{\bm k^l}S_ih_l, S_ih_m\rangle\\
& = \sum_{l,m=1}^r\langle 
W^{ * (\bm k^l-\bm k^m)_{-}}W^{ (\bm k^l-\bm k^m)_{+}}S_ih_l, S_ih_m\rangle\\
&= \sum_{l,m=1}^r\langle 
T^{ * (\bm k^l-\bm k^m)_{-}}T^{ (\bm k^l-\bm k^m)_{+}}S_ih_l, S_ih_m\rangle\ [\text{by regular dilation}]\\
&=  \sum_{l,m=1}^r\langle 
T^{ * (\bm k^l-\bm k^m)_{-}}T^{ (\bm k^l-\bm k^m)_{+}}h_l, h_m\rangle,\
[\text{by doubly commuting property}]
\end{align*}
and by a similar calculation we also have 
\[
\big\|\sum_{l=1}^r W^{\bm k^l} h_l\big\|^2= \sum_{l,m=1}^r\langle 
T^{ * (\bm k^l-\bm k^m)_{-}}T^{ (\bm k^l-\bm k^m)_{+}}h_l, h_m\rangle.
\]
This shows that for all $h_1,\dots,h_r\in \clh$ and $\bm k^1,\dots,\bm k^r\in\mathbb Z^2$,
\[
\big\|\sum_{l=1}^r W^{\bm k^l}S_i h_l\big\|^2=\big\|\sum_{l=1}^r W^{\bm k^l} h_l\big\|^2
\]
for any $1\le i\le d$. Hence, by the minimality of the unitary dilation ~\eqref{minimality}, we have isometry $U_i: \clk\to \clk$
defined by 
\[
U_i(W^{\bm k}h)=W^{\bm k}S_ih\quad(h\in\clh, \bm k\in\mathbb{Z}^2),
\]
for all $i=1,\dots,d$. It is easy to see that 
\[U_i|_{\clh}= S_i, \ U_iU_j=U_jU_i \text{ and } U_i W_m=W_mU_i \]
for all $i,j=1,\dots,d$ and $m=1,2$. 
Consequently, for all $\bm\alpha\in\Z^2$ and for all $(m_1,\ldots,m_d)\in\Z^d$,
	\[
	P_{\clh}W^{\bm\alpha}U_1^{m_1}\cdots U_d^{m_d}|_{\clh}=P_{\clh}W^{\bm\alpha}|_{\clh}(S_1^{m_1}\cdots S_d^{m_d})=T^{\bm\alpha}S_1^{m_1}\cdots S_d^{m_d}.
	\]	
Thus $(W_1,W_2,U_1,\dots, U_d)$ is an isometric dilation of 
$(T_1,T_2,S_1,\dots, S_d)$. It remains to show that $(U_1,\dots, U_d)$
is a co-extension of $(S_1,\dots, S_d)$.
To this end,  for all $h,h'\in\clh$ and $\bm k\in\mathbb{Z}^2$,
	\[
	\langle U_i^*h, W^{\bm k}h'\rangle=\langle h, P_{\clh}W^{\bm k}S_ih'\rangle=\langle h, P_{\clh}W^{*\bm k_-}W^{\bm k_+}S_ih'\rangle=\langle h, T^{*\bm k_-}T^{\bm k_+}S_ih'\rangle, \ (1\le i\le d)
	\]
and using the fact that $S_i$ doubly commutes with $T$, we have
	\[
	\langle h, T^{*\bm k_-}T^{\bm k_+}S_ih'\rangle=\langle h, S_iT^{*\bm k_-}T^{\bm k_+}h'\rangle=\langle S_i^*h, P_{\clh} W^{*\bm k_-}W^{\bm k_+}h'\rangle=\langle S_i^*h,  W^{\bm k}h'\rangle\ (1\le i\le d).
	\]
Thus, $U_i^*h=S_i^*h$, for all $h\in\clh$ and for all $i=1,\dots,d$. 
Hence the proof follows.
\end{proof}


\section{Isometric dilations and von Neumann inequality for $\mathfrak{B}^n_{p,q}(\clh)$}\label{class}
The aim of this section is to construct isometric dilations for the class
$\mathfrak{B}^{n}_{p,q}(\clh)$ for some $1\le p<q\le n$.
For simplicity we fix $p=1,q=n$ and construct isometric dilations for the class $\mathfrak{B}^n_{1,n}(\clh)$. This simplification is harmless as by a suitable rearrangement any operator tuple in $\mathfrak{B}^{n}_{p,q}(\clh)$ can be viewed as a member of $\mathfrak{B}^{n}_{1,n}(\clh)$.
First we introduce some notations which will be followed throughout
the paper. 

For $T\in\clt^n(\clh)$ and for each non-empty order subset $G=\{j_1,\ldots,j_r\}$ of $\{1,\ldots,n\}$, we define
\[
T(G):=(T_{j_1},\ldots,T_{j_r}).
\] 
With the above notation, if $T\in \mathfrak{B}^n(\clh)$ then 
$T(G)\in  \mathfrak{B}^{|G|}(\clh)$ for any non-empty subset $G$ of $\{1,\dots,n\}$, and therefore we define the corresponding 
defect operator and defect spaces as 
\[
D_{T,G}:=\mathbb{S}_r^{-1}(T(G),T(G)^*)^{1/2}=\Big(\sum_{F\subset G}
(-1)^{|F|} T_{F} T_{F}^*\Big)^{1/2}\quad\text{and}\quad \cld_{T,G}:=\overline{\ran}\,D_{T,G}.
\]
For $T\in\clt^n(\clh)$, we also set
\[
\hat{T}_{1n}:=(T_1T_n,T_2,\ldots,T_{n-1}),
\]
the $(n-1)$-tuple in $\clt^{n-1}(\clh)$ obtained from $T$ by removing $T_n$ and replacing $T_1$ by $T_1T_n$. Next we observe 
a crucial property of $\hat{T}_{1n}$ when $T\in \mathfrak{B}^n_{1,n}(\clh)$. A similar result is also observed in ~\cite[Lemma 5.1]{BDHS} in the context of 
Szeg\"o positivity.
\begin{lemma}\label{combine}
Let $T=(T_1,\ldots,T_n)\in\fbn(\clh)$. Then  $\hat{T}_{1n}=(T_1T_n,T_2,\ldots,T_{n-1})\in\mathfrak{B}^{n-1}(\clh)$.
\end{lemma}
\begin{proof}
First note that for $G\subset \{1,\ldots,n-1\}$ with $1\notin G$,
\[
\sum_{F\subset G}
(-1)^{|F|} (\hat{T}_{1n})_{F} (\hat{T}_{1n})_{F}^*= \sum_{F\subset G}
(-1)^{|F|} (\hat{T}_n)_{F} (\hat{T}_n)_{F}^*\geq 0.
\]
On the other hand, for $G\subset \{1,\ldots,n-1\}$ with $1\in G$, it can be checked 
that 
\begin{equation}\label{defect identity 1}
\sum_{F\subset G}
(-1)^{|F|} (\hat{T}_{1n})_{F} (\hat{T}_{1n})_{F}^*= \sum_{F\subset G}
(-1)^{|F|} (\hat{T}_1)_{F} (\hat{T}_1)_{F}^*+ T_n\big( \sum_{F\subset G}
(-1)^{|F|} (\hat{T}_n)_{F} (\hat{T}_n)_{F}^*\big)T_n^*\ge 0.
\end{equation}
This completes the proof.
\end{proof}

The key observation in the above lemma is the identity ~\eqref{defect identity 1} which we use repeatedly in this article. We rewrite the identity 
in terms of defect operators as follows. For $T\in \mathfrak{B}^n_{1,n}(\clh)$ and $G\subset \{1,\ldots,n-1\}$ with $1\in G$ we have
\begin{equation}\label{defect identity 2}
D^2_{\hat{T}_n,G}+T_1D^2_{\hat{T}_1,G}T_1^*=
D^2_{\hat{T}_{1n},G}=D^2_{\hat{T}_1,G}+T_nD^2_{\hat{T}_n,G}T_n^*.
\end{equation}
By the above lemma, if $T\in\fbn(\clh)$ then $\hat{T}_{1n}\in\mathfrak{B}^{n-1}(\clh)$ and therefore for all $G=\{j_1,\ldots,j_{|G|}\}\subset\{1,\ldots,n-1\}$, $\hat{T}_{1n}(G)\in\mathfrak{B}^{|G|}(\clh)$. We denote the corresponding canonical dilation map, as in 
~\eqref{szegodil}, of 
$\hat{T}_{1n}(G)$ by $\Pi_G$, that is 
$\Pi_G: \clh\to H^2_{\cld_{\hat{T}_{1n},G}}(\D^{|G|})$ defined by 
\begin{equation}\label{dilG}
(\Pi_Gh)(\z)=\sum_{\bm{k}\in\Z^{|G|}}\z^{\bm{k}}\otimes D_{\hat{T}_{1n},G}\hat{T}_{1n}(G)^{*\bm{k}}h\quad (h\in\clh, \z\in\D^{|G|})
\end{equation}
such that 
\begin{align*}
\Pi_GT_{j_i}^*=M_{z_i}^*\Pi_G
\end{align*}
for all $i=1,\dots,|G|$.
In particular if  $G=\{1=j_1,\dots, j_{|G|}\}$, then 
\begin{align*}
\Pi_GT_{1}^*T_n^*=M_{z_1}^*\Pi_G
\quad
\text{and}\quad\Pi_GT_{j_i}^*=M_{z_i}^*\Pi_G
\end{align*}
for all $i=2,\dots,|G|$.
The next lemma will be the key to factorize $M_{z_1}$ further as 
a product of $M_{\Phi}$ and $M_{\Psi}$ so that we can split the 
first intertwining relation above into 
\[\Pi_G T_1^*= M_{\Phi}^*\Pi_G\ \text{ and } \Pi_GT_n^*=M_{\Psi}^*\Pi_G\]
where $\Phi$ and $\Psi$ are inner multipliers on $\D^{|G|}$ with their 
product being $z_1$. Such a pair of multipliers $(\Phi, \Psi)$ is known as BCL pair and it turns out that the only way one can factorize a shift is through BCL pair (see ~\cite{BCL} and see ~\cite{DSS} for its counterpart in the context of pure contractions).

\begin{lemma}\label{factorization}
Let $T=(T_1=RS, T_2,\ldots, T_n)\in\mathfrak{B}^n(\clh)$ for some 
commuting contractions $R,S\in \mathcal B (\clh)$ and let $G\subset\{1,\ldots,n\}$. Suppose that there exist bounded operators $F_1$ and $F_2$ on $\clh$, Hilbert spaces $\clf_1$ and $\clf_2$ with $\clf_i\supseteq \overline{\ran}\,F_i$ $(i=1,2)$, a Hilbert space $\cle$,
an isometry $\Gamma_G:\cld_{T,G}\to \cle$ and unitaries 
\[
U_i=\begin{bmatrix}
A_i & B_i\\
C_i & 0
\end{bmatrix}: \cle\oplus\clf_i\to\cle\oplus\clf_i \quad (i=1,2)
\]
satisfying 
\[U_1(\Gamma_GD_{T,G}h,F_1R^*S^*h)=(\Gamma_GD_{T,G}R^*h,F_1h)\]
and
\[U_2(\Gamma_GD_{T,G}h,F_2R^*S^*h)=(\Gamma_GD_{T,G}S^*h,F_2h)\]
for all $h\in\clh$, then 
\begin{align*}
(I\otimes \Gamma_G)\Pi_GR^*=M^*_{\Phi_1}(I\otimes \Gamma_G)\Pi_G\quad
\text{and}\quad(I\otimes \Gamma_G)\Pi_GS^*=M^*_{\Phi_2}(I\otimes \Gamma_G)\Pi_G
\end{align*}
where $\Phi_i(\z)= A_i^*+z_1C_i^*B_i^*$ $(\z\in\D^{|G|})$ is the transfer 
function of the unitary $U_i^*$ for all $i=1,2$ and $\Pi_G$ is the canonical dilation map of $T(G)$, as in ~\eqref{dilG}. 
	
\end{lemma}
\begin{proof}
Because of the symmetric roles of $R$ and $S$, we only prove
that $(I\otimes \Gamma_G)\Pi_G R^* = M_{\Phi_1}^* (I\otimes \Gamma_G)\Pi_G$. Let
$\bm{l} \in \Z^{|G|}$ and $\eta\in\cle$. Then for all $h\in\clh$
\[
\begin{split}
\langle (I\otimes \Gamma_G)\Pi_G R^*h, {\z}^{\bm{l}} \otimes \eta \rangle & = \langle (I \otimes \Gamma_G) \sum_{\bm{k} \in \Z^{|G|}}
{\z}^{\bm{k}} \otimes D_{T,G} T(G)^{* \bm{k}} R^* h,
{\z}^{\bm{l}} \otimes \eta \rangle\\
& = \langle \Gamma_G D_{T,G} T(G)^{* \bm{l}} R^* h, \eta\rangle.
\end{split}
\]
Since
\[
U_1(\Gamma_G D_{T,G} h, F_1 R^*S^*h)=(\Gamma_GD_{T,G} R^*h, F_1h), \quad (h\in\clh) 
\]
we have
\[
\Gamma_G D_{T,G} R^* = A_1 \Gamma_G D_{T,G} + B_1 F_1 R^*S^*,\ \text{and }
F_1 = C_1 \Gamma_G D_{T,G}.
\]
Combining these together, we get
\[
\Gamma_G D_{T,G} R^* = A_1 \Gamma_G D_{T,G} + B_1 C_1 \Gamma_G D_{T,G}
R^*S^*.
\]
This implies that 
\[
\begin{split}
\langle M_{\Phi_1}^*(I\otimes \Gamma_G)\Pi_G h, {\z}^{\bm{l}} \otimes \eta \rangle
& = \langle (I\otimes \Gamma_G)\Pi_G h, M_{\Phi_1}({\z}^{\bm{l}} \otimes \eta)
\rangle
\\
& = \langle (I \otimes \Gamma_G) \sum_{\bm{k} \in \Z^{|G|}}
{\z}^{\bm{k}} \otimes D_{T,G} T(G)^{* \bm{k}} h, (A_1^* +
z_{1} C_1^* B_1^*) ({\z}^{\bm{l}} \otimes \eta) \rangle
\\
& = \langle (A_1 \Gamma_G D_{T,G}+ B_1 C_1 \Gamma_G D_{T,G}R^*S^*) T(G)^{* \bm{l}} h, \eta \rangle 
\\
&= \langle \Gamma_G D_{T,G} T(G)^{* \bm{l}} R^* h, \eta\rangle.
\end{split}
\]
Thus $(I\otimes \Gamma_G)\Pi_G R^* = M_{\Phi_1}^* (I\otimes \Gamma_G)\Pi_G$. This completes the proof.
		
\end{proof}

Our construction of isometric dilations for tuples in $\mathfrak{B}^n_{1,n}(\clh)$ relies on getting co-extensions of the tuples first. On the other hand, construction of these co-extensions is build upon obtaining certain operator tuples corresponding to each $G\subset \{1,\dots,n-1\}$. 
In the next two lemmas we construct these operator tuples. The first lemma deals with the case when $1\in G$.

\begin{lemma}\label{1inG}
Let $T=(T_1,\ldots,T_n)\in\mathfrak{B}^n_{1,n}(\clh)$.
Let $G=\{m_1,\ldots,m_r\}\subset\{1,\ldots,n-1\}$ with $1=m_1\in G$ and set $\bar{G}=\{1,\ldots,n-1\}\smallsetminus G$. Suppose that $T_j$ is a co-isometry for all $j\in\bar{G}$ and $\Pi_G$ be as in (\ref{dilG}). Then there exist a Hilbert space $\clh_{G}$, an isometry $\Gamma_G:\cld_{\hat{T}_{1n},G}\to \clh_{G} $ and $n$-tuple of contractions
$V=(V_1,\ldots,V_n)$ on $H^2_{\clh_{G}}(\D^r)$ such that $\hat{V}_1$ and
$\hat{V}_n$ are doubly commuting and 
\[
 (I\otimes \Gamma_G)\Pi_GT_i^*=V_i^*(I\otimes \Gamma_G)\Pi_G,\
 (1\le i\le n)
\]
where $\Pi_G: \clh\to H^2_{\cld_{\hat{T}_{1n},G}}(\D^r)$ is the dilation map of $\hat{T}_{1n}(G)$ and
\[
\begin{split}
 V_1=P_{H^2_{\clh_{G}}(\D^r)}M_{\Phi}|_{H^2_{\clh_{G}}(\D^r)}, V_n=P_{H^2_{\clh_{G}}(\D^r)}M_{\Psi}|_{H^2_{\clh_{G}}(\D^r)}, V_{m_i}=M_{z_i} (1< i\le r)
 \end{split}
 \]
 and 
 \[
 V_i= (I\otimes W_i),\ (i\in \bar{G})
\]
for inner functions $\Phi$ and $\Psi$, depend only on $z_1$ variable, in $H^{\infty}_{\clb(\clk_G)}(\D^r)$ with $\Phi(\z)\Psi(\z)=z_1$ and a Hilbert space $\clk_G$ containing $\clh_{G}$ and for unitaries $W_i$ on $\clh_{G}$.

Moreover, one also has
\[
(I\otimes \Gamma_G)\Pi_GT_1^*=M_{\Phi}^*(I\otimes \Gamma_G)\Pi_G,\
 \text{and } 
 (I\otimes \Gamma_G)\Pi_GT_n^*=M_{\Psi}^*(I\otimes \Gamma_G)\Pi_G.
\]
	
\end{lemma}

\begin{proof}
We first define several unitaries which will make a way to define 
the $n$-tuple of contractions $(V_1,\dots,V_n)$.
Since $T_j$ is a co-isometry, then
\[
 T_jD^2_{\hat{T}_1,G}T_j^*=\sum_{F\subset G}(-1)^{|F|}(\hat{T}_1(G))_FT_jT_j^*(\hat{T}_1(G))_F^*=D^2_{\hat{T}_1,G}
\]
and 
\[
T_jD^2_{\hat{T}_n,G}T_j^*=\sum_{F\subset G}(-1)^{|F|}(\hat{T}_n(G))_FT_jT_j^*(\hat{T}_n(G))_F^*=D^2_{\hat{T}_n,G},
\]
for all $j\in\bar{G}$. Then by Douglas' lemma, for all $j\in \bar{G}$, there exist
co-isometries $W_{j,1} : \cld_{\hat{T}_1,G}\to\cld_{\hat{T}_1,G}$
and $W_{j,n} : \cld_{\hat{T}_n,G}\to\cld_{\hat{T}_n,G}$
such that
\[
W_{j,1}^{*}D_{\hat{T}_1,G}h=D_{\hat{T}_1,G}T_j^*h\ \text{and }
W_{j,n}^{*}D_{\hat{T}_n,G}h=D_{\hat{T}_n,G}T_j^*h\quad (h\in\clh).
\]
Let the commuting tuple of unitaries $\{W_{j,n}': j\in \bar{G}\}$ on $\clk_{n,G}\supseteq \cld_{\hat{T}_n,G}$
be the minimal unitary co-extension of the tuple of commuting co-isometries $\{W_{j,n}:j\in\bar{G}\}$ and let the commuting tuple of unitaries $\{W_{j,1}^{'}:j\in \bar{G}\}$
on $ \clk_{1,G}\supseteq\cld_{\hat{T}_1,G}$ be the minimal unitary co-extension of the commuting tuple of co-isometries $\{W_{j,1}:j\in\bar{G}\}$.
Set 
\[W_j':=W_{j,n}'\oplus W_{j,1}^{'}\ (j\in\bar{G})\
 \text{and } \clk_G^{'}:=\clk_{n,G}\oplus \clk_{1,G}.
\]
 Then the commuting tuple of unitaries $\{W_j^{'}:j\in\bar{G}\}$ on $\clk_G^{'}$ is the minimal unitary co-extension of $\{W_{j,n}\oplus W_{j,1}:j\in\bar{G}\}$.
Now the identity, as noted in ~\eqref{defect identity 2}, 
\[
D^2_{\hat{T}_n,G}+T_1D^2_{\hat{T}_1,G}T_1^*=
D^2_{\hat{T}_{1n},G}=D^2_{\hat{T}_1,G}+T_nD^2_{\hat{T}_n,G}T_n^*
\]
implies that there exist an isometry $\Gamma_G: \cld_{\hat{T}_{1n},G}\to\clk_G^{'}$ 
and a unitary  
\[
U':\clq_G:=\{(D_{\hat{T}_n,G}h,D_{\hat{T}_1,G}T_1^*h):h\in\clh\}\to
\tilde{\clq}_G:=\{(D_{\hat{T}_n,G}T_n^*h,D_{\hat{T}_1,G}h):h\in\clh\}
\] 
such that 
\begin{equation}\label{VG}
\Gamma_G(D_{\hat{T}_{1n},G}h)=
(D_{\hat{T}_n,G}h, D_{\hat{T}_1,G}T_1^*h),\quad (h\in\clh)
\end{equation}
and
\[
U'(D_{\hat{T}_n,G}h,D_{\hat{T}_1,G}T_1^*h)=
(D_{\hat{T}_n,G}T_n^*h,D_{\hat{T}_1,G}h)\quad (h\in\clh).
\]
By the construction, note that $\clq_G$ and $\tilde{\clq}_G$
are joint $\{W_j^{' *}:j\in \bar{G}\}$-invariant subspaces and 
a straightforward calculation shows that $U^{'}$ intertwines 
the tuple of isometries $\{W_j^{' *}|_{\clq_G}: j\in \bar{G}\}$
and $\{W_j^{' *}|_{\tilde{\clq}_G}: j\in \bar{G}\}$, that is 
\begin{equation}\label{U'}
U'W_j^{' *}|_{\clq_G}=(W_j^{' *}|_{\tilde{\clq}_G})U'\ (j\in\bar{G}).
\end{equation}
Let $\clh_G\subseteq \clk_G^{'}$ and $\tilde{\clh}_G\subseteq \clk_G^{'}$
be the smallest joint $\{W_j^{'}:j\in\bar{G}\}$ reducing subspaces containing $\clq_G$ and $\tilde{\clq}_G$, respectively. 
More precisely,
\[
\clh_G=\bigvee_{\bm k\in\Z^{|\bar{G}|},  q\in\clq_G}\prod_{j\in\bar{G}}W_j^{k_j}q
\ \text{ and }
\tilde{\clh}_G=\bigvee_{\bm k\in\Z^{|\bar{G}|},  \tilde{q}\in\tilde{\clq}_G}\prod_{j\in\bar{G}}W_j^{k_j}\tilde{q}.
\]
Then $\{W_j^{' *}|_{\clh_G}: j\in \bar{G}\}$ and $\{W_j^{' *}|_{\tilde{\clh}_{G}}: j\in\bar{G}\}$ are the minimal unitary extension 
of $\{W_j^{' *}|_{\clq_G}:j\in \bar{G}\}$ and
 $\{W_j^{' *}|_{\tilde{\clq}_G}:j\in \bar{G}\}$, respectively.
By a well-known intertwining lifting theorem, we extend $U^{'}$ 
to a unitary 
\[U^{''}: \clh_G\to \tilde{\clh}_G\ \text{ satisfying } U^{''}|_{\clq_G}=U^{'}\]  and
\begin{equation}\label{intertwiner}
U^{''}W_j^{' *}|_{\clh_G}=(W_j^{' *}|_{\tilde{\clh}_G})U^{''}= W_j^{' *}U^{''}\ (j\in\bar{G}).
\end{equation}
By adding an infinite dimensional Hilbert space $\clk$ if necessary,
we extend $U^{''}$ further to get a unitary 
\begin{equation}\label{unitary}
U:\clk_G \to \clk_G\  \text{ such that } U|_{\clh_G}=U'',
\end{equation}
where $\clk_G:=\clk^{'}_G\oplus \clk $. 
We set 
\[W_j:= W_j^{'}|_{\clh_G}\]
for all $j\in \bar{G}$.
The stage is 
set and we now proceed to find the contractions with
 appropriate properties.
 
First, we define  
 \[V_j:=I_{H^2(\D^r)}\otimes W_j\quad (j\in\bar{G}).\] 
Then $V_j$ is a unitary on $H^2_{\clh_G}(\D^r)$ and observe that 
for all $h\in\clh$,
\begin{align*}
(I\otimes W_j^{*})(I\otimes \Gamma_G)\Pi_Gh
&=(I\otimes W_j^{*})
(I\otimes \Gamma_G)\sum_{\bm{k}\in\Z^r}\z^{\bm{k}}\otimes D_{\hat{T}_{1n},G}\hat{T}_{1n}(G)^{*\bm{k}}h\\
&=(I\otimes W_j^{*})\sum_{\bm{k}\in\Z^r}\z^{\bm{k}}\otimes
(D_{\hat{T}_n,G}\hat{T}_{1n}(G)^{*\bm{k}}h,\,D_{\hat{T}_1,G}\hat{T}_{1n}(G)^{*\bm{k}}T_1^*h)\\
&=\sum_{\bm{k}\in\Z^r}\z^{\bm{k}}\otimes(D_{\hat{T}_n,G}\hat{T}_{1n}(G)^{*\bm{k}}T_j^*h,\,D_{\hat{T}_1,G}\hat{T}_{1n}(G)^{*\bm{k}}T_1^*T_j^*h)\\
&=(I\otimes \Gamma_G)\Pi_GT_j^*h.
\end{align*}
Therefore,  
\begin{equation}\label{jinbarG}
(I\otimes \Gamma_G)\Pi_GT_j^*=V_j^*(I\otimes \Gamma_G)\Pi_G.
\end{equation}
for all $j\in \bar{G}$.
We now proceed towards finding contractions corresponding to $j\in G$. 
Since  $\Pi_G$ is the canonical dilation map of $\hat{T}_{1n}(G)$,
\begin{align*}
\Pi_GT_1^*T_n^*=M_{z_1}^*\Pi_G\
\text{and}\ \Pi_GT_{m_i}^*=M_{z_i}^*\Pi_G\ (1<i\le r),
\end{align*}
where $(M_{z_1},\dots, M_{z_r})$ is the $r$-tuple of shifts on $H^2_{\cld_{\hat{T}_{1n}, G}}(\D^r)$.
Then setting 
\[V_{m_i}:=M_{z_i}\ \text{ on } 
H^2_{\clh_G}(\D^r)\quad (1<i\le r)
\]
 and using the commuting property of $M_{z_i}^*$ and $(I\otimes \Gamma_G)$ we have,
\[
(I\otimes \Gamma_G)\Pi_G T_1^*T_n^*=M_{z_1}^*(I\otimes \Gamma_G)\Pi_G\
\text{and }
(I\otimes \Gamma_G)\Pi_GT_{m_i}^*=V_{m_i}^*(I\otimes \Gamma_G)\Pi_G\quad
(1<i\le r).
\]
Next we factor $M_{z_1}$, using Lemma~\ref{factorization}, to obtain contractions corresponding to $j=1, n$ as follows.
Recall that $\clk_G=\clk_{n,G}\oplus\clk_{1,G}\oplus\clk$. 
Let $P:\clk_G\to \clk_{1,G}\oplus \clk$ be the projection map and $\iota_1 : \clk_{n,G} 
\hookrightarrow \clk_G$ and $\iota_2: \clk_{1,G}\oplus\clk
\hookrightarrow \clk_G$ be the inclusion maps defined by
\[
\iota_1(h)=(h, 0,0),\ \text{and }
\iota_2(k,k')=(0,k,k'),\quad (h\in \clk_{n,G}, k\in\clk_{1,G}, k'\in\clk).
\]
 Then it is easy to see that
\[
\left[
\begin{array}{cc}
P & \iota_{1} \\
\iota_{1}^* & 0 \\
\end{array}
\right] : \clk_G \oplus \clk_{n,G} \raro \clk_G \oplus \clk_{n,G},
\]
is a unitary, and therefore 
\[
U_1 = \left[
\begin{array}{cc}
U^* & 0 \\
0 & I \end{array}
\right]
\left[
\begin{array}{cc}
P & \iota_{1} \\
\iota_{1}^* & 0
\end{array}
\right]= \left[\begin{array}{cc}
U^*P& U^*\iota_1\\
\iota_1^* & 0
\end{array}\right]: \clk_G \oplus \clk_{n,G} \raro \clk_G \oplus \clk_{n,G}
\]
is also a unitary, where $U$ is the unitary as in ~\eqref{unitary}.
We now claim that $U_1$ satisfies the hypothesis of
Lemma~\ref{factorization} with $T= \hat{T}_{1n}(G)$ and  $F_1=D_{\hat{T}_n,G}$. Indeed, for $h \in \clh$,
\begin{align*}
U_{1}(\Gamma_GD_{\hat{T}_{1n},G} h, D_{\hat{T}_n,G} T_{1}^*T_{n}^*h) & =
U_{1}( D_{\hat{T}_n,G}h, D_{\hat{T}_1,G}T_{1}^*h,0_{\clk},  D_{\hat{T}_n,G}T_{1}^*T_{n}^*h)
\\
&=(U^*(D_{\hat{T}_n,G} T_{1}^*T_{n}^*h, D_{\hat{T}_1,G} T_{1}^*h,0_{\clk}), D_{\hat{T}_n,G} h)
\\
&=(D_{\hat{T}_n,G}T_1^* h,D_{\hat{T}_1,G} T_1^{* 2} h,0_{\clk}, D_{\hat{T}_n,G} h)
\\
&=(\Gamma_G D_{\hat{T}_{1n},G} T_1^* h, D_{\hat{T}_n,G} h).
\end{align*}
Similarly, one can check that the unitary
\[
U_{2} =\left[ \begin{array}{cc}
P^\perp & \iota_{2} \\
\iota_{2}^* & 0
\end{array}  \right] \left[ \begin{array}{cc}
U & 0 \\0 & I
\end{array} \right]=\left[ \begin{array}{cc}
P^\perp U & \iota_{2} \\
\iota_{2}^*U & 0 \\
\end{array}  \right]:  \clk_G \oplus (\clk_{1,G}\oplus\clk) \raro \clk_G \oplus (\clk_{1,G}\oplus\clk),
\]
 satisfies  
\[
U_2 (\Gamma_G D_{\hat{T}_{1n},G} h, D_{\hat{T}_1,G} T_1^* T_n^* h,0_{\clk})
=(\Gamma_G D_{\hat{T}_{1n},G} T_n^*
h, D_{\hat{T}_1,G}h,0_{\clk}),
\]
for all $h \in \Hil$. Therefore by Lemma~\ref{factorization}, we
have 
\[
(I\otimes \Gamma_G)\Pi_GT_1^*=M_{\Phi}^*(I\otimes \Gamma_G)\Pi_G\
\text{and }
(I\otimes \Gamma_G)\Pi_GT_n^*=M_{\Psi}^*(I\otimes \Gamma_G)\Pi_G,
\]
where
\[
\Phi(\z)= (P+ z_1 P^{\perp})U\
\text{and }
\Psi(\z)= U^*(P^{\perp}+ z_1P),\quad (\z\in\D^r)
\]
are inner multipliers, depend only on $z_1$ variable, with
\[
\Phi (\z) \Psi(\z) = \Psi(\z) \Phi(\z) = z_1 I_{\clk_G},
\]
for all $\z\in\D^{r}$. Since $\overline{\mbox{ran}} (I\otimes \Gamma_G)\Pi_G\subseteq H^2_{\clh_G}(\D^r)$, we also have 
\[
(I\otimes \Gamma_G)\Pi_GT_1^*=V_1^*(I\otimes \Gamma_G)\Pi_G\
\text{and }
(I\otimes \Gamma_G)\Pi_GT_n^*=V_n^*(I\otimes \Gamma_G)\Pi_G,
\]
where 
\[V_1=P_{H^2_{\clh_G}(\D^r)}M_{\Phi}|_{H^2_{\clh_G}(\D^r)}\
\text{ and } V_n=P_{H^2_{\clh_G}(\D^r)}M_{\Psi}|_{H^2_{\clh_G}(\D^r)}.
\]
We pause for a moment and make a remark that even though 
$M_{\Phi}$ and $M_{\Psi}$ commute each other, $V_1$ and $V_n$
does not necessarily commute. The reader must have observed that 
we have obtained the $n$-tuple of contractions $(V_1,\dots,V_n)$
on $H^2_{\clh_G}(\D^r)$ with the required intertwining property. 
The proof will be complete if we show that $\hat{V}_1$ and $\hat{V}_n$
are doubly commuting. We only show that $\hat{V}_n$ is doubly commuting as the proof for $\hat{V}_1$ is similar. 

Since $M_{\Phi}$ and $M_{z_i}$ on $H^2_{\clk_G}(\D^r)$ doubly commute each other for all $i=2,\dots,r$, it follows that $V_1$ and 
$V_{m_i}$ doubly commute for all $i=2,\dots,r$. It is obvious that $V_{m_i}$ doubly commutes with $V_j$ for all $i=2,\dots,r$ and $j\in\bar{G}$. Thus it remains to show that $V_1$ commutes with $V_j$ for all $j\in \bar{G}$. 
To this end, we first claim that 
\[
P_{\clh_G}(PU)W_j^*=W_j^*P_{\clh_G}(PU|_{\clh_G})\ \text{ and }
P_{\clh_G}(P^{\perp}U)W_j^*=W_j^*P_{\clh_G}(P^{\perp}U|_{\clh_G})
\] 
for all $j\in\bar{G}$. Indeed, since $\clh_G$ is a joint $\{W_j^{' *}:j\in\bar{G}\}$ reducing subspace then $P_{\clh_G}$ commutes with $W_j^{' *}$ for all $j\in \bar{G}$. Also by the construction 
\[
PW_j^{' *}|_{\clk_G^{'}}=W_j^{' *}P|_{\clk_G^{'}}.
\]
Then, using the intertwining property ~\eqref{intertwiner}, we have  for all $h\in\clh_G$ and $j\in\bar{G}$,
\begin{align*}
P_{\clh_G}(PU)W_j^*h=P_{\clh_G}(PU^{''})W_j^{' *}h
=P_{\clh_G}PW_j^{' *}U^{''}h=W_j^{' *}P_{\clh_G} PUh= W_j^{*}P_{\clh_G} PUh.
\end{align*}
This proves the first identity in the claim and the proof of the 
second identity is similar and we left it for the reader. 
These identities, in turn, implies that
\begin{align*}
P_{H^2_{\clh_G}(\D^r)}M_{\Phi}|_{H^2_{\clh_G}(\D^r)}(I\otimes W_j^*)|_{H^2_{\clh_G}(\D^r)}=&P_{H^2_{\clh_G}(\D^r)}M_{(P+z_1P^{\perp})U}(I\otimes W_j^*)|_{H^2_{\clh_G}(\D^r)}\\
=&P_{H^2_{\clh_G}(\D^r)}M_{(PUW_j^*+z_1P^{\perp}UW_j^*)}|_{H^2_{\clh_G}(\D^r)}\\
=&P_{H^2_{\clh_G}(\D^r)}(I\otimes W_j^*)P_{H^2_{\clh_G}(\D^r)}M_{(PU+z_1P^{\perp}U)}|_{H^2_{\clh_G}(\D^r)}\\
=&(I\otimes W_j^*)P_{H^2_{\clh_G}(\D^r)}M_{\Phi}|_{H^2_{\clh_G}(\D^r)}.
\end{align*} 
Thus $V_1$ commutes with $V_j^*$ and consequently, by Fuglede-Putnam, $V_1$ commutes with $V_j$ for all $j\in\bar{G}$. 
This completes the proof.

\end{proof}
\begin{remark}\label{remark1}
It should be noted that the canonical dilation map $\Pi_G$ is not an isometry and that is why the $n$-tuple  of contractions $V$ is not a co-extension of $T$. A word of caution is in order regarding the $n$-tuple 
contractions $V=(V_1,\dots,V_n)$. The tuple $V$ is not a commuting 
tuple, in general, and the only trouble is that $V_1$ and $V_n$ do not commute each other. However,  $(V_0=M_{z_1}, V_1,\dots, V_{n-1})$ on $H^2_{\clh_G}(\D^r)$ is 
an $n$-tuple of commuting contractions with 
\[
  (I\otimes \Gamma_G)\Pi_GT_1^*T_n^*=M_{z_1}^*(I\otimes \Gamma_G)\Pi_G,\
 \text{and }(I\otimes \Gamma_G)\Pi_GT_i^*=V_i^*(I\otimes \Gamma_G)\Pi_G\
 (1\le i\le n-1),
\]
and also
\begin{equation}\label{the identity}
V_1^*V_0=V_n\ \text{and } V_n^*V_0=V_1.
\end{equation}
Similar statement also can be made for the tuple 
$(V_0=M_{z_1}, V_2, \dots, V_n)$. 
\end{remark}

The situation for the case when $1\notin G$ is much simpler 
and we consider it in the next lemma. 
\begin{lemma}\label{1notinG}
Let $T=(T_1,\ldots,T_n)\in\mathfrak{B}^n_{1,n}(\clh)$.
Let $G=\{m_1,\ldots,m_r\}\subset\{1,\ldots,n-1\}$ with $1\notin G$ and set $\bar{G}=\{1,\ldots,n-1\}\smallsetminus G$.
Suppose $T_j$ is a co-isometry for all $j\in\bar{G}\cup\{n\}$. Then there exist a Hilbert space $\clh_G\supset \cld_{\hat{T}_{1n},G}$  and an $n$-tuple of commuting isometries 
$V=(V_1,\ldots,V_n)$ on $H^2_{\clh_G}(\D^r)$ such that $\hat{V}_1$ and $\hat{V}_n$ are doubly commuting and 
\[
 \Pi_GT_{i}^*=V_i^*\Pi_G\ (1\le i\le n)
\]
where $\Pi_G: \clh\to H^2_{\cld_{\hat{T}_{1n},G}}(\D^r)
\subset H^2_{\clh_G}(\D^r)$ is the canonical dilation map of $\hat{T}_{1n}(G)$ and
\[
 V_{m_i}=M_{z_i} (1\leq i\leq r)\ \text{and } V_j=I\otimes U_j\ (j\in \bar{G}\cup\{n\})
\]
for some commuting unitaries $U_j$'s on $\clh_G$.

\end{lemma}

\begin{proof}
	Since $T_j$ $(j\in\bar{G})$ and $T_n$ are co-isometries, a straight
forward computation as done in the proof of Lemma~\ref{1inG} yields 
	\[
	T_jD_{T,G}^2T_j^*=D_{T,G}^2\ \text{and } T_nD_{T,G}^2T_n^*=D_{T,G}^2.
	\]
	Then, by Douglas' lemma, there exist co-isometries $S_j:\cld_{T,G}\to\cld_{T,G}$ $(j\in\bar{G})$ and $S_n:\cld_{T,G}\to\cld_{T,G}$
	such that 
	\[
	S_j^*D_{T,G}h=D_{T,G}T_j^*h\ \text{and } S_n^*D_{T,G}h=D_{T,G}T_n^*h,
	\]
for all $h\in\clh$. Clearly $\{S_j: j\in\bar{G}\cup\{n\}\}$ is a tuple of commuting co-isometries on $\cld_{T,G}$. Let the commuting tuple of unitaries $\{U_j: j\in\bar{G}\cup \{n\}\}$ on $\clh_G\supseteq\cld_{T,G}$ be the minimal unitary co-extension of $\{S_j: j\in\bar{G}\cup\{n\}\}$.
On the other hand, since $\Pi_G:\clh\to H^2_{\cld_{T,G}}(\D^r)\subset H^2_{\clh_G}(\D^r)$ is the canonical dilation map of $\hat{T}_{1n}(G)$
then 
\[
	\Pi_GT_{m_i}^*=M_{z_i}^*\Pi_G
\]
for all $i=1\ldots,r$. Also it is evident from the construction of $U_j$'s that 
\[
\Pi_GT_{j}^*=(I\otimes U_j)^*\Pi_G
\] 
for all $j\in \bar{G}\cup\{n\}$. Set 
\[V_{m_i}:=M_{z_i}\ (1\leq i\leq r)\ \text{ and } V_j:=I\otimes U_j\ (j\in\bar{G}\cup\{n\}).\]
 Then the $n$-tuple of commuting isometries $V=(V_1,\dots, V_n)$ has the
 required property. This completes the proof. 
\end{proof}
We need one more lemma which describes a canonical way to construct co-isometries out of commuting contractions. 
\begin{lemma}\label{tilde}
Let $T=(T_1,\ldots,T_n)\in\T^n(\clh)$. Let $G\subset\{1,\ldots,n\}$ and $\bar{G}=\{1,\dots,n\}\smallsetminus G$. Then there exist a positive operator $Q:\clh \to\clh $ and contractions $\tilde{T}_j:\overline{\ran}\,Q\to\overline{\ran}\,Q$ $(1\le j\le n) $ defined by 
	\[
	\tilde{T}_j^*Qh=Q{T}_j^*h,\quad(h\in\clh)
	\]
such that $\tilde{T}_j$ is a co-isometry for all $j\in\bar{G}$.
\end{lemma}
\begin{proof}
Since $T_{\bar{G}}$ is a contraction, strong operator limit (SOT) of $T_{\bar{G}}^nT_{\bar{G}}^{*n}$ exists as $n\to\infty$.
Set
	\[	Q^2:=\text{SOT-}\lim_{n\to\infty}T_{\bar{G}}^nT_{\bar{G}}^{*n}.
	\]
Since $T_j$ is a contraction, it is easy to see that 
	\[
	T_jQ^2T_j^*\leq Q^2,
	\]
 for all $j=1\ldots,n$. Consequently by Douglas' lemma, for each $1\le j\le n$, there exists a contraction  $\tilde{T}_j:\overline{\mbox{ran}}\,Q\to\overline{\ran}\,Q$ such that
	\[
	\tilde{T}_j^*Qh=Q{T}_j^*h \quad(h\in\clh).
	\]
	Moreover, for all $j\in\bar{G}$, 
	\[
	Q^2\geq T_jQ^2T_j^*=\text{SOT-}\lim_{n\to\infty}T_{\bar{G}}^nT_jT_j^*T_{\bar{G}}^{*n}\geq \text{SOT-}\lim_{n\to\infty}T_{\bar{G}}^{n+1}T_{\bar{G}}^{*(n+1)}=Q^2,
	\]
	that is
	\[
	T_jQ^2T_j^*=Q^2.
	\]
Hence $\tilde{T}_j$ is a co-isometry for all $j\in\bar{G}$. This completes the proof.
\end{proof}

Combining the above lemmas we now find co-extensions for 
$n$-tuples in $\mathfrak{B}^n_{1,n}(\clh)$.
\begin{propn}\label{predil}
Let $T=(T_1,\dots,T_n)\in\mathfrak{B}^n_{1,n}(\clh)$. Then for each $G\subseteq\{1,\ldots,n-1\}$ there exist a Hilbert space $\clh_G$ and contractions $V_{G,1},\ldots,V_{G,n}$ on $H^2_{\clh_G}(\D^{|G|})$ 
 such that 
	\[
	\Pi T_j^*=\big(\bigoplus_{G}V_{G,j}^*\big)\Pi, \ (1\le j\le n)
	\]
where $\Pi:\clh\to\bigoplus_{G}H^2_{\clh_G}(\D^{|G|})$ is an isometry 
and by convention $H^2_{\clh_{\emptyset}}(\D^{|\emptyset|}):=\clh_{\emptyset}$.

Moreover, if we set $V_j:=\bigoplus_{G}V_{G,j}$ for all $j=1,\dots,n$
and $V=(V_1,\dots,V_n)$
then $\hat{V}_1$ and $\hat{V}_n$ are doubly commuting and $V_j$ is an isometry for all $j=2,\dots, n-1$.
	
	
%
\end{propn}
\begin{proof}
As usual, we shall work on the $(n-1)$-tuple 
 $\hat{T}_{1n}\in\mathfrak{B}^{n-1}(\clh)$. Let us fix $G=\{m_1,\ldots,m_{|G|}\}\subseteq\{1,\ldots,n-1\}$ and set $\bar{G}=\{1,\ldots,n-1\}\smallsetminus G$. Let $Q_{G}:\clh\to\clh$ be the positive operator defined by
	\[
	Q_{G}^2:=\text{SOT-}\lim_{m\to\infty}(\hat{T}_{1n}(\bar{G}))^m(\hat{T}_{1n}(\bar{G}))^{* m}.
	\]
Then, by Lemma~\ref{tilde}, we have a contraction
 $S_j:\overline{\ran}\,Q_G \to \overline{\ran}\,Q_G$ defined by 
\[
S_j^*Q_Gh=Q_G{T}_j^*h,\quad(h\in\clh)
\]
for all $j=1,\dots,n$ so that $S_j$ is a co-isometry 
for all $j\in\bar{G}$. In the case when $1\in \bar{G}$, both $S_1$ and $S_n$ are co-isometries because their product $S_1S_n$ is a co-isometry. We claim that 
 \[S:=(S_1,\ldots,S_n)\in\mathfrak{B}^n_{1,n}(\overline{\ran}\,Q_G).\] Indeed, the claim follows from the fact that 
if $\sum_{F\subset G'}
(-1)^{|F|} (\hat{T}_{1n})_{F} (\hat{T}_{1n})_{F}^*\ge 0$ for some 
$G'\subset \{1,\dots, n-1\}$ then
\[
\sum_{F\subset G'}
(-1)^{|F|} (\hat{T}_{1n})_{F}Q_G^2 (\hat{T}_{1n})_{F}^*=
\text{SOT-}\lim_{m\to\infty} \hat{T}_{1n}(\bar{G})^m\Big(\sum_{F\subset G'}
(-1)^{|F|} (\hat{T}_{1n})_{F} (\hat{T}_{1n})_{F}^*\Big) \hat{T}_{1n}(\bar{G})^{* m}\ge 0,
\] 
and the later positivity is equivalent to the positivity of $\sum_{F\subset G'}
(-1)^{|F|} (\hat{S}_{1n})_{F} (\hat{S}_{1n})_{F}^*$.
Next we apply Lemma~\ref{1inG} and Lemma~\ref{1notinG} for the 
tuple $S:=(S_1,\ldots,S_n)\in\mathfrak{B}^n_{1,n}(\overline{\ran}\,Q_G)$, to obtain the building blocks of the co-extension 
we are after. We consider the following two cases.

\textbf{Case I}:
Suppose $1=m_1\in G$. Since $S_j$ $(j\in\bar{G})$ are co-isometries, then  by Lemma \ref{1inG}, there exist
a Hilbert space $\clh_G$, an isometry $\Gamma_G: \cld_{\hat{S}_{1n},G}\to\clh_G$ and $n$-tuple of contractions
	$V=(V_{G,1},\ldots,V_{G,n})$ on $H^2_{\clh_{G}}(\D^{|G|})$ such that $\hat{V}_1$ and $\hat{V}_n$ are doubly commuting
 and 
	\[
	(I\otimes \Gamma_G)\tilde{\Pi}_GS_i^*=V_{G,i}^*(I\otimes \Gamma_G)\tilde{\Pi}_G,\
	(1\le i\le n)
	\]
where $\tilde{\Pi}_G:\overline{\mbox{ran}}\,Q_G\to H^2_{\cld_{\hat{S}_{1n}, G}}(\D^{|G|})$ is the canonical dilation map of $\hat{S}_{1n}(G)$.
Now using the identity $S_j^*Q_G=Q_GT_j^*$ and setting \[\Pi_G:=\tilde{\Pi}_GQ_G,\] we have that 
	\begin{equation}\label{int1}
	(I\otimes \Gamma_G)\Pi_GT_i^*=V_{G,i}^*(I\otimes \Gamma_G)\Pi_G,\
	(1\le i\le n)
	\end{equation}
where $(I\otimes \Gamma_G)\Pi_G :\clh \to H^2_{\clh_G}(\D^{|G|})$ 
is a contraction with 
\begin{align*}
	\|(I\otimes \Gamma_G)\Pi_Gh\|^2 &=
	\lim_{\bm{k}\to\infty}\sum_{F\subset G}(-1)^{|F|}\big\|\big(\hat{S}_{1n}\big)_F^{*\bm{k}}Q_Gh\big\|^2\\
& =\lim_{\bm{k}\to\infty}\sum_{F\subset G}(-1)^{|F|}\|Q_G(\hat{T}_{1n})_F^{* \bm{k}}h\|^2,
\end{align*}
for all $h\in\clh$.

\textbf{Case II}: Suppose $1\notin G$. In such a case, $S_{j}$ is a co-isometry for all $j\in\bar{G}\cup\{n\}$ and therefore, by Lemma \ref{1notinG}, there
exist a Hilbert space $\clh_G$  and commuting $n$-tuple of isometries 
$V'=(V_{G,1},\ldots,V_{G,n})$ on $H^2_{\clh_G}(\D^{|G|})$ such that $\hat{V'}_1$ and $\hat{V'}_n$ are doubly commuting and 
	\[
	\tilde{\Pi}_GS_{i}^*=V_{G,i}^*\tilde{\Pi}_G\ (1\le i\le n)
	\]
	where $\tilde{\Pi}_G: \overline{\mbox{ran}} Q_G \to H^2_{\clh_G}(\D^{|G|})$ is the canonical dilation map of $\hat{S}_{1n}(G)$. Again using the 
identity $S_j^*Q_G=Q_GT_j^*$ and setting \[\Pi_G:=\tilde{\Pi}_GQ_G,\] we get 
\begin{equation}\label{int2}
	\Pi_GT_i^*=V_{G,i}^*\Pi_G,\
	(1\le i\le n)
	\end{equation}
and for all $h\in\clh$,
	\[
	\|\Pi_Gh\|^2=\lim_{\bm{k}\to\infty}\sum_{F\subset G}(-1)^{|F|}\|Q_G(\hat{T}_{1n})_F^{ * \bm{k}}h\|^2.
	\]

Now we combine all the intertwining maps, obtained in the above two cases, together to obtain a co-extension.  Define
$\Pi: \clh \to\bigoplus_{G}H^2_{\clh_G}(\D^{|G|})$ by 
\[
\Pi(h)(G)=
\left  \{  
    \begin{array}{cc}
      (I\otimes \Gamma_G)\Pi_Gh & 1\in G \\
      \Pi_Gh & \text{otherwise} 
    \end{array}
\right. , \quad  (h\in\clh, G\subset \{1,\dots, n-1\}) \] 
where $\Pi(h)(G)$ denotes the $G$-th coordinate of $\Pi(h)$.
Then by ~\eqref{int1} and ~\eqref{int2}, it follows that 
\[
\Pi T_j^*=(\bigoplus_{G}V_{G,j}^* )\Pi
\]
for all $j=1,\dots,n$. Note that if $\Pi$ is an isometry then the above identity gives 
a co-extension of $T$. 
To show $\Pi$ is an isometry, for any $h\in\clh$, we compute  
\begin{align*}
	\|\Pi h\|^2=&\sum_{G\subset\{1,\ldots,n-1\}}\|\Pi_G h\|^2\quad(\text{as $\Gamma_G$'s are isometry})\\
	=&\sum_{G\subset\{1,\ldots,n-1\}}\lim_{\bm{k}\to\infty}\sum_{F\subset G}(-1)^{|F|}\|Q_G(\hat{T}_{1n})_F^{* \bm{k}}h\|^2\\
	=&\sum_{G\subset\{1,\ldots,n-1\}}\sum_{F\subset G}(-1)^{|F|}\lim_{\bm{k}\to\infty}\|T_{\bar{G}}^{* \bm{k}}(\hat{T}_{1n})_F^{* \bm{k}}h\|^2\\
	=&\sum_{A\subset\{1,\ldots,n-1\}}\lim_{  \bm{k}\to\infty}\|T_{A}^{* \bm{k}}h\|^2\sum_{F\subset A}(-1)^{|F|}.
	\end{align*}
	For each subset $A\neq\emptyset$ one can see that $\sum_{F\subset A}(-1)^{|F|}=0$. So $\|\Pi h\|^2=\|h\|^2$ for all $h\in\clh$ and hence $\Pi$ is an isometry. The moreover part is now clear from the construction 
of $V_{G,j}$'s. This completes the proof.
\end{proof}
We make several important remarks about the above proposition and these observations will be used to prove the main theorem below.

\begin{remarks}\label{crucial remark}
\textup{(i)} The main drawback of the above proposition is that the co-extension $V$ of $T$ is not a commuting tuple of contractions (see Remark~\ref{remark1} ). As noted earlier, the only problem is that $V_1$ and $V_n$ do not commute each other. However, there are several things which are nice and help us to work further to find 
isometric dilation of $T$. For example, $V_j$ is an isometry for all $j=2,\dots, n-1$ and if we set 
\[
V_0=\bigoplus_{G}V_{G,0}\in
 \clb\big(\bigoplus_{G}H^2_{\clh_G}(\D^{|G|})\big),\  \text{where } V_{G,0}=\left\{
\begin{array}{cc}
M_{z_1} &  1\in G\\
V_{G,1}V_{G,n} & \text{otherwise}
\end{array}
\right .  ,
\]
then $V_0$ is an isometry, $(V_0,V_1,\dots,V_{n-1})$ is an $n$-tuple of commuting
 contractions with 
 \[
 \Pi T_1^*T_n^*= V_0^*\Pi\ \text{and } \Pi T_j^*= V_j^*\Pi\ (1\le j\le n-1).
 \]
 and, in view of ~\eqref{the identity},
 \begin{equation}\label{the identity II}
 V_1^*V_0=V_n\  \text{ and } V_n^*V_0=V_1.
 \end{equation}
Another crucial fact is that $(V_1,\dots, V_{n-1})$ and $(V_0, V_2,\dots, V_{n-1})$ are doubly commuting. In other words, the doubly commuting tuple $(V_2,\dots, V_{n-1})$ doubly commute with both $V_0$ and $V_1$.
Needless to say that a similar statement also can be made about the 
commuting tuple of contractions $(V_0,V_2,\dots, V_n)$.

 \textup{(ii)} If $\clq:=\ran\, \Pi$, then $\clq$ is a joint $(V_0^*, V_1^*,\dots,V_n^*)$-invariant subspace of $\bigoplus_{G}H^2_{\clh_G}(\D^{|G|})$, $(P_{\clq}V_1|_{\clq}, \dots, P_{\clq}V_{n}|_{\clq})$ is an
 $n$-tuple of  commuting contractions and 
 \[
 ( T_1,\dots, T_{n})\cong
  (P_{\clq}V_1|_{\clq},\dots, P_{\clq}V_{n}|_{\clq}).
 \]
 Moreover, 
 \[(P_{\clq}V_1|_{\clq})( P_{\clq}V_{n}|_{\clq})=(P_{\clq}V_n|_{\clq})( P_{\clq}V_{1}|_{\clq})=P_{\clq}V_0|_{\clq}.\]
 \end{remarks}
 We now find isometric dilations of operator tuples in $\mathfrak{B}^n_{1,n}(\clh)$, which is the main theorem of this article.
\begin{thm}\label{dil}
Let $T=(T_1,\ldots,T_n)\in\mathfrak{B}^n_{1,n}(\clh)$. Then $T$ has an isometric dilation $W=(W_1,\ldots,W_n)$ such that $\hat{W}_1$ and $\hat{W}_n$ are $*$-regular isometric dilations of $\hat{T}_1$ and $\hat{T}_n$, respectively.
\end{thm}
\begin{proof}
Let $\Pi$,
$(V_{G,1},\ldots,V_{G,n})$ on $H^2_{\clh_G}(\D^{|G|})$
and $V=(V_1,\dots, V_n)$ be as in Proposition~\ref{predil}. Then $V=(V_1,\dots,V_n)$ is a co-extension of $T$, that is 
\[
\Pi T_j^*= V_j^*\Pi \quad (1\le j\le n)
\]
where $V_j=\bigoplus_{G}V_{G,j}$ for all $j=1,\dots,n$. We set 
\[
V_{0}:=\bigoplus_{G}V_{G,0}\in
 \clb\big(\bigoplus_{G}H^2_{\clh_G}(\D^{|G|})\big),\  
 \text{where } V_{G,0}=\left\{
\begin{array}{cc}
M_{z_1} &  1\in G\\
V_{G,1}V_{G,n} & \text{otherwise}
\end{array}
\right . .
\]
As we have noted in Remark~\ref{crucial remark}, 
$(V_0, V_1\dots, V_{n-1})$ is a commuting tuple with $V_i$'s are isometry 
except $V_1$ and the doubly commuting tuple $(V_2,\dots, V_{n-1})$ doubly commute with both $V_0$ and $V_1$.  Since $V_0$
is an isometry, then the pair $(V_0, V_1)$ has a regular unitary dilation 
$(W_0, W_1)$ on $\clk$. Then by Lemma~\ref{doubly commuting lifting},  we extends $V_j$ to an isometry $W_j$ on $\clk$, for all $j=2,\dots,n-1$, such that the tuple $(W_0,W_1,\dots, W_{n-1})$ on $\clk$ is an isometric dilation 
of $(V_0, V_1,\dots, V_{n-1})$ on $\bigoplus_{G}H^2_{\clh_G}(\D^{|G|})$. 
Set 
\[
W_n:=W_1^*W_0.
\]
Then clearly $(W_1,\dots, W_n)$ is a commuting tuple of isometries. We
claim that $(W_1,\dots, W_n)$ is an isometric dilation of $(P_{\clq}V_1|_{\clq}, \dots, P_{\clq}V_{n}|_{\clq})$, where $\clq=\mbox{ran}\, \Pi$.
To prove the claim, let $k=(k_1,\dots,k_n)\in \mathbb Z^n_{+}$ and 
let us denote $\clx:=\bigoplus_{G}H^2_{\clh_G}(\D^{|G|})$. We divide 
the proof of the claim in the following two cases.  

\textbf{Case I}:
If $k_n\ge k_1$, then
\begin{align*}
P_{\clx}W_1^{k_1}\cdots W_n^{k_n}|_{\clx} & =P_{\clx}W_0^{k_n}W_1^{* (k_n-k_1)}\cdots W_{n-1}^{k_{n-1}}|_{\clx}\\
& =( P_{\clx}W_{1}^{* (k_n-k_1)}W_0^{k_n}|_{\clx})(W_2^{k_2}\cdots W_{n-1}^{k_{n-1}}|_{\clx} )\ \  [W_j\clx\subset \clx, j=2,\dots,n-1]\\
&=V_1^{* (k_n-k_1)}V_0^{k_n}V_2^{k_2}\cdots V_{n-1}^{k_{n-1}}\ \ [(W_0,W_1) \text{is a regular dilation of } (V_0,V_1)]\\
&=(V_1^{*}V_0)^{(k_n-k_1)}V_0^{k_1} V_2^{k_2}\cdots V_{n-1}^{k_{n-1}}\\
& =V_n^{(k_n-k_1)}V_0^{k_1}V_2^{k_2}\cdots V_{n-1}^{k_{n-1}},\ [\text{by } ~\eqref{the identity II}]
\end{align*}
and therefore using the fact that $\clq\subset \clx$ is a joint
 $(V_0^*, V_1^*,\dots,V_n^*)$-invariant subspace
\begin{align*}
P_{\clq}W_1^{k_1}\cdots W_n^{k_n}|_{\clq}&=P_{\clq}V_n^{(k_n-k_1)}V_0^{k_1}V_2^{k_2}\cdots V_{n-1}^{k_{n-1}}|_{\clq}\\
&=(P_{\clq}V_n|_{\clq})^{(k_n-k_1)}( P_{\clq} V_0|_{\clq})^{k_1}
(P_{\clq}V_2|_{\clq})^{k_2}\cdots (P_{\clq}V_{n-1}|_{\clq})^{k_{n-1}}\\
&=(P_{\clq}V_1|_{\clq})^{k_1}\cdots (P_{\clq}V_{n}|_{\clq})^{k_{n}}.
\end{align*}
Here for the last equality we have used that $P_{\clq}V_0|_{\clq}=(P_{\clq}V_1|_{\clq})(P_{\clq}V_n|_{\clq})$.

\textbf{Case II}: If $k_1> k_n$, then 
\begin{align*}
P_{\clx}W_1^{k_1}\cdots W_n^{k_n}|_{\clx} & =P_{\clx}W_0^{k_n}W_1^{k_1-k_n}\cdots W_{n-1}^{k_{n-1}}|_{\clx}\\
&=V_0^{k_n}V_1^{k_1-k_n}\cdots V_{n-1}^{k_{n-1}},
\end{align*}
and therefore,
\begin{align*}
P_{\clq}W_1^{k_1}\cdots W_n^{k_n}|_{\clq}&
=P_{\clq}V_0^{k_n}V_1^{k_1-k_n}\cdots V_{n-1}^{k_{n-1}}|_{\clq}\\
&=(P_{\clq}V_1|_{\clq})^{k_1}\cdots (P_{\clq}V_n|_{\clq})^{k_n}.
\end{align*}
This proves the claim. On the other hand, we have already observed in Remark~\ref{crucial remark}
that
\[
 ( T_1,\dots, T_{n})\cong
  (P_{\clq}V_1|_{\clq},\dots, P_{\clq}V_{n}|_{\clq}).
\]
Hence $(W_1,\dots, W_n)$ is an isometric dilation of $(T_1,\dots, T_n)$. Since the tuple $(W_2,\dots, W_{n-1})$ is a co-extension of $(V_2,\dots, V_{n-1})$ and $(V_2,\dots, V_{n-1})$ is a co-extension of 
$(T_2,\dots, T_{n-1})$, $(W_2,\dots, W_{n-1})$ is a co-extension of 
$(T_2,\dots, T_{n-1})$. Finally since $\hat{W}_1$ and $\hat{W}_n$ are
tuples of doubly commuting isometries, it follows that $\hat{W}_1$ and $\hat{W}_n$ are $*$-regular isometric dilation of $\hat{T}_1$ and $\hat{T}_n$, respectively. This completes the proof.


\end{proof}	
Few remarks are in order.

\begin{remarks}
\textup{(i)} The converse of the above theorem is true. That is, if  $T\in\clt^n(\clh)$ has an isometric dilation $W\in\clt^n(\clk)$ so that 
$\hat{W}_1$ and $\hat{W}_n$ are $*$-regular isometric dilation 
of $\hat{T}_1$ and $\hat{T}_2$ respectively, then $T\in\mathfrak{B}^n_{1,n}(\clh)$. This immediately follows from the
fact that an operator tuple satisfies Brehmer positivity if and only if it 
has a $*$-regular isometric dilation.

\textup{(ii)} If $T\in\mathfrak{B}^n_{p,q}(\clh)$ then interchanging $T_p$ with $T_1$ and $T_q$ with $T_n$ we can assume $T\in\mathfrak{B}^n_{1,n}(\clh)$. So, Theorem \ref{dil} provides dilations for tuples in $\mathfrak{B}^n_{p,q}(\clh)$.
	
\textup{(iii)} If $T\in\mathfrak{B}^n_{1,n}(\clh)$ such that $\h_n$ is pure then $T$ is also a member of the class $\clt^n_{1,n}(\clh)$ (see ~\eqref{tnpq}) considered in ~\cite{BDHS}.  Note that in this case $\hat{T}_{1n}$ is also a pure tuple. 
Then the positive operator $Q_G$, defined in the proof of Proposition~\ref{predil}, is $0$ for all $G\subsetneq \{1,\dots,n-1\}$ and 
 $Q_{G}=I_{\clh}$ for $G=\{1,\dots,n-1\}$. This implies $\Pi_G=0$
 for all $G\subsetneq\{1,\dots,n-1\}$ and $\Pi_G$ is an isometry 
 for $G=\{1,\dots,n-1\}$. Then it follows from Lemma~\ref{1inG} and 
 Proposition~\ref{predil} that the commuting tuple of isometries 
 $ (M_{\Phi}, M_{z_2},\dots, M_{z_{n-1}}, M_{\Psi}) $ on 
 $H^2_{\clk_G}(\D^{n-1})$ is a co-extension of $T$. Thus we recover 
 the isometric dilations obtained in ~\cite{BDHS} for such tuples. 
\end{remarks}

 We end the section with the von Neumann inequality for $\mathfrak{B}^n_{p,q}(\clh)$. Recall that, if an $n$-tuple of commuting contraction has an isometric dilation then it satisfies von Neumann inequality. So we have the following theorem as an immediate consequence of Theorem~\ref{dil}.
\begin{thm}
	Let $T\in\mathfrak{B}^n_{r,q}(\clh)$ with $1\le r<q\le n$. Then, 
	\[
	\|p(T)\|_{\clb(\clh)} \leq \sup_{\z \in \D^n} |p(\z)|
	\]
	for all $p\in\mathbb{C}[z_1,\ldots,z_n]$.
\end{thm}

\vspace{0.1in} \noindent\textbf{Acknowledgement:}  The research of the second named author is supported by
DST-INSPIRE Faculty Fellowship No. DST/INSPIRE/04/2015/001094.

\bibliographystyle{plain}

\end{document}